\documentclass[a4paper,12pt]{amsart}

\usepackage{amssymb}
\usepackage{latexsym}
\usepackage{amsfonts}
\usepackage{amsmath}
\usepackage{eucal}
\usepackage{bm}
\usepackage{bbm}
\usepackage{graphicx}
\usepackage[english]{varioref}
\usepackage[nice]{nicefrac}
\usepackage[all]{xy}
\usepackage{amsthm}

\newcommand{\supp}{\text {\rm supp}}

\newcommand{\Ad}{\text {\rm Ad}}

\def\i{^{-1}}
\def\ge{\geqslant}
\def\le{\leqslant}
\def\<{\langle}
\def\>{\rangle}
\def\lto{\hookrightarrow}

\def\a{\alpha}
\def\b{\beta}
\def\g{\gamma}
\def\G{\Gamma}
\def\d{\delta}

\def\e{\epsilon}

\def\o{\omega}

\def\s{\sigma}
\def\t{\tau}
\def\th{\theta}

\def\l{\lambda}
\def\z{\zeta}

\def\ZZ{\mathbb Z}
\def\NN{\mathbb N}
\def\QQ{\mathbb Q}
\def\FF{\mathbb F}
\def\RR{\mathbb R}

\def\kk{\mathbf k}

\def\aff{\text aff}
\def\tch{\tilde \ch}

\def\ch{\mathcal H}

\def\cp{\mathcal P}

\def\ts{{\tilde s}}

\def\tW{\tilde W}
\def\tw{{\tilde w}}

\def\tch{\tilde \ch}

\def\fR{\mathfrak R}
\def\fH{\mathfrak H}

\def\ux{\underline{x}}

\def\subset{\subseteq}

\theoremstyle{plain}
\newtheorem{thm}{Theorem}[section]
\newtheorem*{thm*}{Theorem}
 \newtheorem{prop}[thm]{Proposition}
 \newtheorem{lem}[thm]{Lemma}
 \newtheorem{cor}[thm]{Corollary}

\theoremstyle{definition}

\theoremstyle{remark}

\newtheorem*{rmk}{Remark}
\newtheorem*{claim*}{Claim}

\begin{document}

\author{Xuhua He}
\address{Department of Mathematics, University of Maryland, College Park, MD 20742, USA}
\email{xuhuahe@math.umd.edu}
\thanks{X. H. was partially supported by NSF DMS-1463852.}
\author{Sian Nie}
\address{Institute of Mathematics, Academy of Mathematics and Systems Science, Chinese Academy of Sciences, 100190, Beijing, China}
\email{niesian@amss.ac.cn}
\title[Pro-$p$ Hecke algebras]{Cocenters and representations of pro-$p$ Hecke algebras}

\begin{abstract}
In this paper, we study the relation between the cocenter $\overline{\tch}$ and the representations of an affine pro-$p$ Hecke algebra $\tch=\tch(0, -)$. As a consequence, we obtain a new criterion on the supersingular representation: a (virtual) representation of $\tch$ is supersingular if and only if its character vanishes on the non-supersingular part of the cocenter $\overline{\tch}$.
\end{abstract}

\maketitle

\section*{Introduction}

\subsection{} Let $G$ be a $p$-adic group and $\tW$ be its Iwahori-Weyl group. The Iwahori-Hecke algebra $\tch_q$ is a deformation of the group algebra of $\tW$. It plays an important role in the study of the ordinary representations of $G$.

For representations of $G$ in characteristic $p$ (the defining characteristic), one  expects that there is a close relation between the mod-$p$ representations of $G$ and of the pro-$p$ Hecke algebra $\tch$ of $G$. The pro-$p$ Hecke algebra is a deformation of the group algebra $\tW(1)$, with parameter $q=0$. Here $\tW(1)$ is the pro-$p$ Iwahori-Weyl group, an extension of $\tW$ by a finite torus.

\subsection{}
For a pro-$p$ Hecke algebra of a $p$-adic group (i.e. the associated group $\tW(1)$ is the pro-$p$ Iwahori-Weyl group of a $p$-adic group), the representations are studied by Abe \cite{A} and Vign\'eras \cite{V14}, based on the Bernstein presentation and Satake-type isomorphism.

In this paper, we study the representations via a different approach, the ``cocenter program''.

Let us first provide some background on the ``cocenter program''.

For a group algebra of a finite group, the cocenter is very simple. The elements in the same conjugacy class of the group have the same image in the cocenter of the group algebra and the cocenter has a standard basis given by the conjugacy classes. There is a perfect pairing between the cocenter of the group algebra and the Grothendieck group of finite dimensional (complex) representations, via the trace map. This is a ``toy model'' for the ``cocenter program''.

For finite or affine Hecke algebras (with nonzero parameters), the cocenter is more complicated. The elements in the same conjugacy class may not have the same image in the cocenter. However, based on some remarkable properties on the minimal length elements in the Weyl group, one may show that the elements of minimal length in a conjugacy class of the Weyl group still have the same image in the cocenter of the Hecke algebra and the cocenter is still indexed by the conjugacy classes of the group.

For finite Hecke algebras (with generic parameters), the relation between the cocenter and the representations is fairly simple. The dimension of the cocenter equals the number of irreducible representations, and the trace map gives a perfect pairing between the cocenter and representations. For affine Hecke algebras, both the dimension of the cocenter and the number of irreducible representations are infinite and the counting-number method does not simply apply. In \cite{CH}, we introduced the rigid cocenter and rigid quotient of Grothendieck group of representations, and shows that they form a perfect pairing under the trace map, and the whole cocenter and all the finite dimensional representations can be understood via the rigid part of the parabolic subalgebras.

\subsection{} For finite and affine pro-$p$ Hecke algebra (with parameter $q=0$), the trace map from the cocenter to the linear functions on the Grothendieck group of finite dimensional representations, is surjective, but not injective. However, the knowledge of the structure of the cocenter still does a big help in the understanding of representations.

We first show in Theorem \ref{cocenter} that

\begin{thm}\label{1}
For finite and affine pro-$p$ Hecke algebra $\tch$, the cocenter is spanned by the image of $T_w$, where $w$ runs over the minimal length elements in its conjugacy class in $\tW(1)$.
\end{thm}

For finite pro-$p$ Hecke algebras, the irreducible representations are just characters. For affine pro-$p$ Hecke algebras, it is easy to construct a family $\{\pi_{J, \G, \Xi, V}\}$ of representations by taking the parabolic induction from characters of the parabolic algebras of $\tch$. See \S\ref{aleph} for the precise definition.

One of the main results in this paper is the following (see Theorem \ref{main-7}).

\begin{thm}
Let $\tch$ be an affine pro-$p$ Hecke algebra. The set $\{\pi_{J, \G, \Xi, V}\}$ is a basis of the Grothendieck group $R(\tch)$.
\end{thm}

For pro-$p$ Hecke algebras of $p$-adic groups, the above result is obtained by Abe \cite{A} and Vign\'eras \cite{V14}. Our strategy is different from Abe and Vign\'eras. We use the structure of the cocenter (see Theorem \ref{cocenter}) and the character formula established in \S\ref{6}.

\subsection{} Among all the representations of the affine pro-$p$ Hecke algebra $\tch$, the supersingular representations are the most important ones. Abe \cite{A}  showed that for pro-$p$ Hecke algebras of $p$-adic groups, any representation can be obtained from supersingular ones by the parabolic inductions. The classification of supersingular representations for affine pro-$p$ Hecke algebras is obtained by Vign\'eras \cite{V14}. In Theorem\ref{5.4}, we give a new proof of the classification (in the more general setting) and we also give a new criterion of supersingular representations.

\begin{thm}\label{ss}
A virtual representation $\pi$ of $\tch$ is supersingular if and only if $Tr(\overline{\tch}^{nss}, \pi)=0$, where $\overline{\tilde \ch}^{nss}$ is the non-supersingular part of the cocenter defined in \S\ref{nss}.
\end{thm}

\subsection{} The paper is organized as follows.

In section 1, we recall the definition of pro-$p$ Hecke algebras and trace maps. In section 2, we study the cocenters of finite and affine pro-$p$ Hecke algebras. In section 3, we discuss the cocenter and representations of finite pro-$p$ Hecke algebras. In section 4, we define the parabolic subalgebras of affine pro-$p$ Hecke algebras. In section 5, we discuss the standard representatives associated to minimal length elements and study their power. Such results is used in section 6, in which we study the character formula for affine pro-$p$ Hecke algebras. Finally, in section 7,
we give a basis of the Grothendieck group of finite dimensional modules of affine pro-$p$ Hecke algebras and gives a new criterion of supersingular representations.

\subsection*{Acknowledgement} After the paper was finished, we learned from Vign\'eras that Theorem \ref{para} is also proved in her paper \cite{Vigxx} and her preprint \cite{Vigx}. We thank her for sending us \cite{Vigx} and for many useful comments. 

\section{Preliminary}

\subsection{}\label{1.1} We start with a sextuple $(W, S, \Omega, \tW, Z, \tW(1))$, where $(W, S)$ is a Coxeter system, $\Omega$ is a group acting on $W$ and stabilizing $S$, $\tW=W \rtimes \Omega$, $Z$ is a finite commutative group, and we have a short exact sequence $$1 \to Z \to \tW(1) \xrightarrow{\pi} \tW \to 1.$$

Let $\ell$ be the length function on $W$. It extends to a length function on $\tW$ by requiring that $\ell(\t)=0$ for $\t \in \Omega$, and inflates to a length function $\ell$ on $\tW(1)$.

Since $Z$ is commutative, the conjugation action of $\tW(1)$ on $Z$ induces an action of $\tW$ on $Z$, which we denote by $\bullet$.

For any subset $D$ of $\tW$, we denote by $D(1)$ the inverse image of $D$ in $\tW(1)$.

Let $w \in W$. The support of $w$ is defined to be the set of simple reflections that appear in some (or equivalently, any) reduced expression of $w$ and is denoted by $\supp(w)$. For $w \in W$ and $\t \in \Omega$, we define $\supp(w \t)$ to be $\cup_{i \in \NN} \t^i(\supp(w))$.   For $\tw \in \tW(1)$, we define $\supp(\tw)=\supp(\pi(\tw))$.

\subsection{}\label{1.2}  Now we recall the definition of generic pro-$p$ Hecke algebra introduced by Vign\'eras in \cite{V14-1}.

Let $T=\cup_{w \in W} w S w \i \subset W$ be the set of reflections in $W$. Let $\kk$ be an algebraically closed field. We choose $(q_t, c_t) \in \kk \times \kk[Z]$ for $t \in T(1)$ such that
\begin{itemize}
\item $q_{w t w\i}=q_t$ for $w \in \tW(1)$ and $q_{t z}=q_t$ for $z \in Z$.

\item $c_{w t w \i}=w \bullet c_t$ for $w \in \tW(1)$ and $c_{t z}=c_t z$ for $z \in Z$.
\end{itemize}

Let $\tch(q, c)$ be the associative $\kk$-algebra with basis $(T_w)_{w \in \tW(1)}$ subject to the following relations \begin{gather*} T_w T_{w'}=T_{w w'}, \quad \text{ for } w, w' \in \tW(1) \text{ with } \ell(w w')=\ell(w)+\ell(w'); \\ T_s^2=q_s T_{s ^2}+c_s T_s, \quad \text{ for } s \in S(1). \end{gather*} We denote by $\ch(q, c)$ the subalgebra of $\tch(q, c)$ spanned by $T_w$ for $w \in W(1)$.

In this paper, we are mainly interested in the case where $q_t \equiv 0$. We simply write $\tch$ for $\tch(0, c)$ and write $\ch$ for $\ch(0, c)$. In this case, the second relation becomes $T_s^2=c_s T_s$ for $s \in S(1)$. The algebra $\tch$ plays an important role in the study of mod-$p$ representations of reductive groups over finite fields of characteristic $p$ and over $p$-adic fields.

In the case where $W$ is a finite Coxeter group, we call $\tch$ a {\it finite pro-$p$ Hecke algebra}. In the case where $W$ is an affine Weyl group, we call $\tch$ an {\it affine pro-$p$ Hecke algebra}.

\subsection{} Let $[\tch, \tch]$ be the commutator of $\tch$, the subspace of $\tch$ spanned by $[T_{w}, T_{w'}]:=T_{w} T_{w'}-T_{w'} T_{w}$ for $w, w' \in \tW(1)$. Let $\overline{\tch}=\tch/[\tch, \tch]$ be the cocenter of $\tch$. Denote by $R(\tch)_\kk$ the ($\kk$-span of the) Grothendieck group of finite dimensional representations of $\tch$ over $\kk$, i.e., the $\kk$-vector space with basis given by the isomorphism classes of irreducible representations of $\tch$. Consider the trace map $$Tr: \overline{\tch} \to R(\tch)_\kk^*, \qquad h \mapsto (V \mapsto Tr(h, V)).$$

Similar maps for affine Hecke algebras with generic nonzero parameters are studied in the joint work of Ciubotaru and the first-named author \cite{CH}. It is proved in \cite{CH} that the trace map is injective and there is a ``perfect pairing'' between the rigid-cocenter and rigid-representations of $\tch_q$.


\section{Cocenter of $\tch$}

\subsection{}\label{2.1} For $w, w' \in \tW(1)$ and $s \in S(1)$, we write $w \xrightarrow{s} w'$ if $w'=s w s \i$ and $\ell(w') \le \ell(w)$.  We write $w \to w'$ if there exists a sequence $w=w_1, w_2, \cdots, w_n=w'$ of elements in $\tW(1)$ such that for any $k$, $w_{k-1} \xrightarrow{s_k} w_k$ for some $s_k \in S(1)$. We write $w \approx w'$ if $w \to \t w' \t \i$ and $\t w' \t \i \to w$ for some $\t \in \Omega(1)$. In this case, we say that $w$ and $w'$ are in the same cyclic-shift class.

\begin{lem}\label{to-min}
Let $w \in \tW(1)$ and $s \in S(1)$.

(1) If $\ell(s w s \i)=\ell(w)$, then $T_{w} \equiv T_{s w s \i} \mod [\tch, \tch]$.

(2) If $\ell(s w s \i)<\ell(w)$, then $T_{w} \equiv c_{s \i} T_{w s} \mod [\tch, \tch]$.
\end{lem}

\begin{proof}
(1) Without loss of generality, we may assume that $\ell(s w)=\ell(w)-1$. Then $$T_{w}=T_{s \i} T_{s w} \equiv T_{s w} T_{s \i}=T_{s w s \i} \mod [\tch,\tch].$$ Here the last equality follows from the fact that $\ell(s w)=\ell(s w s \i)-1$.

(2) We have $\ell(s w s)=\ell(w)-2$. So $$T_{w}=T_{s \i} T_{s w s} T_{s \i} \equiv T^2_{s \i} T_{s w s}=c_{s \i} T_{s \i} T_{s w s} =c_{s \i} T_{w s} \mod [\tch, \tch].$$ Here the last equality follows from the fact that $\ell(s w s)=\ell(s w)-1$.
\end{proof}

The following consequence follows easily from Lemma \ref{to-min} (1).

\begin{cor}\label{to-min2}
Let $w, w' \in \tW(1)$ with $w \approx w'$. Then $$T_{w} \equiv T_{w'} \mod [\tch, \tch].$$
\end{cor}

\begin{proof}
By definition, there exists a sequence $w=w_1, w_2, \cdots, w_n=w'$ such that for any $1<k<n$, $\ell(w_k)=\ell(w_{k-1})$ and $w_k=s w_{k-1} s \i$ for some $s \in S(1)$ and $w'=\t w_{n-1} \t \i$ for some $\t \in \Omega(1)$. By Lemma \ref{to-min}(1), $T_{w} \equiv T_{w_{n-1}} \mod [\tch, \tch]$. By definition, $$T_{w'}=T_{\t} T_{w_{n-1}} T_{\t \i} \equiv T_{w_{n-1}} T_{\t \i} T_{\t}=T_{w_{n-1}} \mod [\tch, \tch].$$ The corollary is proved.
\end{proof}

\smallskip

Let $\tW(1)_{\min}$ be the set of elements in $\tW(1)$ that are of minimal length in their conjugacy classes. We have the following result.

\begin{thm}\label{min}
Assume that $W$ is a finite Coxeter group or an affine Weyl group. Then for any $w \in \tW(1)$, there exists $w' \in \tW(1)_{\min}$ such that $\tw \to \tw'$.
\end{thm}

\begin{proof}
Since the length function on $\tW(1)$ is induced from the length function on $\tW$ through $\pi$, the statement follows directly from \cite[Theorem 1.1]{GP} and \cite[Theorem 2.6]{GKP} (see also \cite{HN12}) if $W$ is a finite Coxeter group, and from \cite[Theorem 2.9]{HN14} if $W$ is an affine Weyl group.
\end{proof}

Now we prove the main result of this section.

\begin{thm}\label{cocenter}
Let $\tch$ be a finite or an affine pro-$p$ Hecke algebra. Then the cocenter $\overline{\tch}$ is spanned by the image of $T_{w}$ for $w \in \tW(1)_{\min}$.
\end{thm}

\begin{proof}
Let $x \in \tW(1)$. We prove by induction that the image of $T_{x}$ in $\overline{\tch}$ is spanned by $T_{w}$ for $w \in \tW(1)_{\min}$.

If $x \in \tW(1)_{\min}$, then the statement is obvious. If $x \notin \tW(1)_{\min}$, then there exists $x' \in \tW(1)$ and $s \in S(1)$ with $x \approx x'$ and $\ell(s x' s \i)<\ell(x')=\ell(x)$. By Corollary \ref{to-min2} and Lemma \ref{to-min} (2), we have $$T_{x} \equiv T_{x'} \equiv c_{s} T_{x' s\i} \mod [\tch, \tch].$$

Note that $\ell(x' s\i)<\ell(x')=\ell(x)$. By inductive hypothesis, the image of $T_{x' s\i}$ in $\overline{\tch}$ is spanned by $T_{w}$ for $w \in \tW(1)_{\min}$. Hence the image of $T_{x}$ in $\overline{\tch}$ is spanned by $T_{w}$ for $w \in \tW(1)_{\min}$.
\end{proof}

\subsection{} Let $\text{Cyc}(\tW(1)_{\min})$ be the set of cyclic-shift classes in $\tW(1)_{\min}$. For $\Sigma \in \text{Cyc}(\tW(1)_{\min})$, we denote by $T_{\Sigma}$ the image of $T_w$ in $\overline{\tch}$ for any $w \in \Sigma$. By Corollary \ref{to-min2}, $T_{\Sigma}$ is well-defined. By Theorem \ref{cocenter}, for a finite or an affine pro-$p$ Hecke algebra, its cocenter is spanned by $(T_{\Sigma})_{\Sigma \in \text{Cyc}(\tW(1)_{\min})}$. It is interesting to see if the spanning set is in fact a basis. It is known to be true for finite $0$-Hecke algebras \cite[Theorem 4.4]{He15} and for affine $0$-Hecke algebras \cite[Theorem 0.1]{HNx}.

\section{Finite pro-$p$ Hecke algebras}

In this section, we assume that $\tch$ is a finite pro-$p$ Hecke algebra and we discuss the relation between the cocenter and representations of $\tch$.

\subsection{}\label{3-1} Recall that $\ch$ is the subalgebra of $\tch$ spanned by $T_w$ for $w \in W(1)$. It is proved by Vign\'eras \cite[Proposition 2.1 \& Proposition 2.2]{V14}, every irreducible representation of $\ch$ is a character and is of the form $\Xi_{\chi, \G}$, where $\chi$ is a character of $Z$ and $\G \subset \{s \in S; \chi(c_s) \neq 0\}$. Here the character $\Xi_{\chi, \G}$ is defined to be $$\Xi_{\chi, \G}(T_s)=\begin{cases} \chi(c_{s}), & \text{ if } s \in \G(1); \\ 0, & \text{ otherwise}. \end{cases}$$

We set $\G_\Xi=\G$ for $\Xi=\Xi_{\chi, \G}$. Let $\Omega(\Xi)(1)$ be the stabilizer of $\Xi$ in $\Omega(1)$. Let $V$ be an irreducible representation of $\Omega(\Xi)(1)$. We say the pair $(\Xi, V)$ is {\it permissible} with respect to $(W(1), \Omega(1))$ if $Z \subseteq \Omega(\Xi)(1)$ acts on $V$ via $\Xi$. Set $$I(\Xi, V)=\text{Ind}^{\tch}_{\ch \otimes_{\kk[Z]} \kk[\Omega(\Xi)(1)]} (\Xi \otimes V).$$

We say that two permissible pairs $(\Xi, V)$ and $(\Xi', V')$ are equivalent if there exists $\g \in \Omega(1)$ such that $({}^\g \Xi, {}^\g V)=(\Xi', V')$. Here ${}^\g \Xi$ (resp. ${}^\g V$) denotes the twisted module of $\ch$ (resp. $\Omega({}^\g \Xi)(1)=\g \Omega(\Xi)(1) \g \i$) by $\g$. In this case, we write $(\Xi, V) \sim (\Xi', V')$. It is proved by Vign\'eras \cite[Proposition 6.17]{V14} that every irreducible representation of $\tch$ is of the form $I(\Xi, V)$ and $I(\Xi, V) \cong I(\Xi', V')$ if and only if $(\Xi, V) \sim (\Xi', V')$.

The following formula follows easily from the definition of induced modules.

\begin{lem} \label{ftrace}
Let $(\Xi, V)$ be a permissible pair. For $w \in W(1)$ and $\t \in \Omega(1)$ we have $$Tr(T_{w \t}, I(\Xi, V))= \sum_{\g \in \Omega(1) / \Omega(\Xi)(1)} {}^\g \Xi(T_w) Tr(\t, {}^\g V).$$ Here we set $Tr(\t, {}^\g V)=0$ if $ \t \notin \Omega({}^\g\Xi)(1)$.
\end{lem}

\subsection{}
We denote by $\tch^{=S} \subseteq \tch$ the $\kk$-linear space generated by $T_{w \t}$, where $w \in W(1)$ with $\supp(w)=S$ and $\t \in \Omega(1)$. Denote by $R(\tch)_{=S}$ the $\kk$-linear space spanned by the simple $\tch$-modules $I(\Xi, V)$, where $(\Xi, V)$ is a permissible pair such that $\G_\Xi=S$.

By Dedekind Theorem, the trace map $Tr: A \to R(A)_\kk^*$ is surjective for any $\kk$-algebra $A$. For finite pro-$p$ Hecke algebra $\tch$, we have the following refinement.

\begin{prop} \label{seperate}
Let $\G \subseteq S$. Then the trace map $Tr: \tch^{=S} \to R(\tch)_{=S} \,^*$ is surjective.
\end{prop}
\begin{proof}
Let $M \in R(\tch)_{=S}$ such that $Tr(\tch^{=S}, M)=0$. We show that $M=0$.

Assume $M=\sum_{[(\Xi, V)]} a_{[(\Xi, V)]} I(\Xi, V)$, where $a_{[(\Xi, V)]} \in \kk$ and $[(\Xi, V)]$ ranges over the $\sim$-equivalence classes of permissible pairs with $\G_\Xi=S$.

Now we show that each coefficient $a_{[(\chi, V)]}$ vanishes. Fix $w_0 \in W(1)$ with $\supp(w_0)=S$. For any $w \in W(1)$ and $\t \in \Omega(1)$, $T_{w_0} T_w $ is a linear combination of $T_{w'}$ with $w' \in W(1)$ such that $\supp(w')=S$. So $Tr(T_{w_0} T_w T_\t, M)=0$ by assumption. Using Lemma \ref{ftrace}, we have
\begin{align*} 0=Tr(T_{w_0} T_w T_\t, M) &= \sum_{[(\Xi, V)]} \sum_{\g \in \Omega(1) / \Omega(\Xi)(1)} a_{[(\Xi, V)]} {}^\g \Xi(T_{w_0}) {}^\g \Xi(T_w) Tr(\t, {}^\g V) \\ &=\sum_{(\Xi', V')} a_{(\Xi', V')} \Xi'(T_{w_0}) Tr(\t, V') \Xi'(T_w),\end{align*} where in the last expression, $(\Xi', V')$ ranges over permissible pairs such that $\G_{\Xi'}=S$, and $a_{(\Xi', V')}=a_{[(\Xi, V)]}$ for $(\Xi', V') \in [(\Xi, V)]$.

Now we regard $\sum_{(\Xi', V')} a_{(\Xi', V')} \Xi'(T_{w_0}) Tr(\t, V') \Xi'(T_w)$ as the virtual character $\sum_{(\Xi', V')} a_{(\Xi', V')} \Xi'(T_{w_0}) Tr(\t, V') \Xi'$ evaluated at $T_w$. Since $w$ runs over all the elements in $W(1)$, we have $$\sum_{(\Xi', V')} a_{(\Xi', V')} \Xi'(T_{w_0}) Tr(\t, V') \Xi'=0.$$

Therefore, for each $\Xi'$,  we have \begin{align*}\sum_{V'; \ \text{$(\Xi', V')$ is permissible}} a_{(\Xi', V')} \Xi'(T_{w_0}) Tr(\t, V')=0. \end{align*} We regard $\sum_{V'; \ \text{$(\Xi', V')$ is permissible}} a_{(\Xi', V')} \Xi'(T_{w_0}) Tr(\t, V')$ as the virtual characters $\sum_{V'; \ \text{$(\Xi', V')$ is permissible}} a_{(\Xi', V')} \Xi'(T_{w_0}) Tr(-, V')$ evaluated at $\t$. Since $\t$ runs over all the elements in $\Omega(\Xi')(1)$, we have $a_{(\Xi', V')} \Xi'(T_{w_0})=0$ for any $V'$. In particular, $a_{(\Xi, V)} \Xi(T_{w_0})=0$. Since $\Xi(T_{w_0}) \neq 0$, we have $a_{[(\Xi, V)]}=a_{(\Xi, V)}=0$ as desried.
\end{proof}

\section{Affine pro-$p$ Hecke algebras and parabolic algebras}

\subsection{} Let $\fR=(X, R, Y, R^\vee, F_0)$ be a based root datum, where $X$ and $Y$ are free abelian groups of finite rank together with a perfect pairing $\< , \>: X \times Y \to \ZZ$, $R \subset X$ is the set of roots, $R^\vee \subset Y$ is the set of coroots and $F_0 \subset R$ is the set of simple roots. Let $\a \mapsto \a^\vee$ be the natural bijection from $R$ to $R^\vee$ such that $\<\a, \a^\vee\>=2$. For $\a \in R$, we denote by $s_\a: X \to X$ the corresponding reflections stabilizing $R$. Let $S_0=\{s_\a; \a \in F_0\}$ be the set of simple reflections of the associated finite Weyl group $W_0$. Let $R^+ \subset R$ be the set of positive roots determined by $F_0$. Let $X^+=\{\l \in X; \<\l, \a^\vee\> \ge 0, \, \forall \a \in R^+\}$. For any $v \in X_\QQ$ , we set $J_v=\{s_\a \in S_0; \<v, \a^\vee\>=0\}$. For any $J \subset S_0$, we set $X^+(J)=\{\l \in X^+; J_\l=J\}$.

\subsection{} Let $W_{\aff}=\ZZ R \rtimes W_0$ be the affine Weyl group and $S_{\aff} \supset S_0$ be the set of simple reflections in $W$. Then $(W_{\aff}, S_{\aff})$ is a Coxeter group. Let $\tW=X \rtimes W_0$ be the extended affine Weyl group. Then $W_{\aff}$ is a subgroup of $\tW$. For $\l \in X$, we denote by $\e^\l \in \tW$ the corresponding translation element.

Let $V=X \otimes_\ZZ \RR$. For $\a \in R$ and $k \in \ZZ$, set $$H_{\a, k}=\{v \in V; \<v, \a^\vee\>=k\}.$$ Let $\fH=\{H_{\a, k}; \a \in R, k \in \ZZ\}$.  Connected components of $V-\cup_{H \in \fH}H$ are called alcoves. Let $$C_0=\{v \in V; 0 < \<v, \a^\vee\> <1, \, \forall \a \in R^+\}$$ be the fundamental alcove. We may regard $W_{\aff}$ and $\tW$ as subgroups of affine transformations of $V$, where $t^\l$ acts by translation $v \mapsto v+\l$ on $V$. The actions of $W_{\aff}$ and $\tW$ on $V$ preserve the set of alcoves.

For any $\tw \in \tW$, we denote by $\ell(\tw)$ the number of hyperplanes in $\fH$ separating $C_0$ from $\tw C_0$. Then $\tW=W_{\aff} \rtimes \Omega$, where $\Omega=\{\tw \in \tW; \ell(\tw)=0\}$ is the subgroup of $\tW$ stabilizing fundamental alcove $C_0$. The conjugation action of $\Omega$ on $\tW$ preserves the set $S_{\aff}$ of simple reflections in $W_{\aff}$.



\subsection{}
Let $Z$ be a finite commutative group and $\tW(1)$ be a group containing $Z$ as a normal subgroup and $\tW(1)/Z \cong \tW$. As in \S\ref{1.2}, we may define the generic affine pro-$p$ Hecke algebra $\tch(q, c)$ and the affine pro-$p$ Hecke algebra $\tch$. The parameters $q_s$ for $s \in S(1)$ gives a multiplicative function $w \mapsto q(w)$ on $\tW(1)$  such that $q(\o)=1$ if $\o \in \Omega(1)$ and $q(s)=q_s$ if $s \in S(1)$.

Examples of such algebras include the pro-$p$ Iwahori-Hecke algebras of reductive $p$-adic groups.

By \cite[Corollary 2]{V05}, the map $T_w \mapsto {}^\iota T_w:=(-1)^{\ell(w)} q(w) T_{w\i}\i$ gives an involution $\iota$ of $\tch_q$. We still denoted by $\iota$ the induced involution of $\tch=\tch(0,c)$.



\subsection{} \label{parabolic}
For any $J \subset S_0$, we denote by $R_J$ the set of roots spanned by $J$ and set $R^\vee_J=\{a^\vee; \a \in R_J\}$. Let $\fR_J=(X, R_J, Y, R^\vee_J, J)$ be the based root datum corresponding to $J$. Let $W_J \subset W_0$ and $\tW_J=X \rtimes W_J$ be the Weyl group and the extended affine Weyl group of $\fR_J$ respectively. We say $\tw \in \tW_J$ is $J$-positive if $\tw \in  t^\l W_J $ for some $\l \in X$ such that $\<\l, \a\> \ge 0$ for $\a \in R^+ \smallsetminus R_J$. Denote by $\tW_J^+$ the set of $J$-positive elements, which is a submonoid of $\tW_J$, see \cite[Section 6]{BK} and \cite[II.4]{V98}.

We set $\fH_J=\{H_{\a, k} \in \fH; \a \in R_J, k \in \ZZ\}$ and $C_J=\{v \in V; 0 < \<v, \a^\vee\> < 1, \a \in R_J^+\}$. For any $\tw \in \tW_J$, we denote by $\ell_J(\tw)$ the number of hyperplanes in $\fH_J$ separating $C_J$ from $\tw C_J$.

Let $(W_J)_{\aff}=\ZZ R_J \rtimes W_J$ and let $J_{\aff} \supseteq J$ be the set of simple reflections of $(W_J)_{\aff}$. Then $\tW_J=(W_J)_{\aff} \rtimes \Omega_J$, where $\Omega_J=\{\tw \in \tW_J; \ell_J(\tw)=0\}$. We denote by $\le_J$ the Bruhat order on $\tW_J$. Note that $\le_J$ differs from the restriction to $\tW_J$ of the Bruhat order on $\tW$.

We denote by $\tW^J$ (resp. ${}^J \tW$) the set of minimal coset representatives in $\tW/W_J$ (resp. $W_J \setminus \tW$). For $J, K \subset S_0$, we simply write $\tW^J \cap {}^K \tW$ as ${}^K \tW^J$. We define ${}^J W_0, W_0^J$ and ${}^J W_0^K$ in a similar way.

The following result is proved in \cite[Lemma 4.1]{A}.

\begin{lem}  \label{alcove}
Let $x, y \in \tW_J(1)$. If $x \in \tW_J^+(1)$ and $y \le_J x$, then $y \in \tW_J^+(1)$.
\end{lem}

\subsection{}
Let $\l \in X(1)$. Then $\l=\l_1 \l_2\i$ for some $\l_1, \l_2 \in X^+(1)$. We set $\th_\l=T_{\l_1} T_{\l_2}\i$. It is easy to see that $\th_\l$ does not depend on the choices of $\l_1$ and $\l_2$. For $\l, \l' \in X(1)$ we have $\th_\l \th_{\l'}=\th_{\l \l'}$. For $u \in \tW_J(1)$, there exist $u' \in  \tW_J^+(1)$ and $\l \in X^+(J)(1)$ such that $u=u' \l\i$. We set $T_u^J=T_{u'} T_\l\i$. It is easy to see that $T_u^J$ does not depend on the choices of $u'$ and $\l$.

Let $\tch_J(q, c)$ be the $\kk$-linear subspace of $\tch(q, c)$ spanned by $T_u^J$ for $u \in \tW_J(1)$. Let $\tch^+_J(q, c)$ be the $\kk$-linear subspace of $\tch(q, c)$ spanned by $T_u=T^J_u$ for $u \in \tW_J^+(1)$. We write $\tch_J$ for $\tch_J(q, c)$ specialized at $q=0$ and $\tch^+_J$ for  $\tch^+_J(q, c)$ specialized at $q=0$.

\begin{thm}\label{para}
Let $J \subset S_0$. Then

(1) The multiplication map on $\tch(q, c)$ gives $\tch_J(q, c)$ a generic pro-$p$ Hecke algebra structure.

(2) The multiplication map on $\tch=\tch(0, c)$ gives $\tch^+_J$ a pro-$p$ Hecke algebra structure.
\end{thm}

\begin{rmk}
We have a natural embedding $$\tch_J^+ \lto \tch, \qquad T_{\tw}^J \mapsto T_{\tw}.$$ Notice that this embedding does not extend to an algebra homomorphism $\tch_J \to \tch$ since $T_\l^J$ for $\l \in X^+(J)(1)$ is invertible in $\tch_J$, but not invertible in $\tch$.
\end{rmk}

If $\tch(q, c)$ is the pro-$p$ Hecke algebra of a $p$-adic group, then $\tch_J$ is the pro-$p$ Hecke algebra of the corresponding Levi subgroup. In this case, the statements are obvious. The general situation requires more work and will be proved in the rest of this section.

\smallskip

We have discuss some relations on $\th$, which essentially follows from \cite[Lemma 2.5 \& 2.7]{L}.
\begin{lem} \label{commutator}
Let $s \in S_0(1)$ and $\chi \in X(1)$. Denote by $\a_s$ the simple root corresponding to $s$.

(1) If $\<\chi, \a_s^\vee\>=0$, then $T_s \th_{\chi}=\th_{s \chi s\i} T_s$.

(2) If $\<\chi, \a_s^\vee\>=1$, then $T_{s\i}\i \th_{\chi} T_{s\i}\i=\th_{s \chi s}$.

(3) If $\<\chi, \a_s^\vee\>=2$, $\chi \in X^+(1)$ and $\a_s^\vee \in 2Y$, then
\begin{itemize}
\item[(i)] $T_{w'} T_{w''} \th_{\chi\i}=T_s \th_{s\i \chi s}$;
\item[(ii)] $T_{w'} T_{\ts} T_{w''} \th_{\chi\i}=\th_{w' \ts w'' \chi\i}$;
\item[(iii)] $T_{w'} T_{\ts\i} T_{\ts} T_{w''} \th_{\chi\i}=\th_\chi T_{s\i}\i$.
\end{itemize}
Here $w'=\l_s s \o$ and $w''=\o\i s\i \l_s\i \chi s \chi$, $\l_s \in X^+(1)$ such that $\pi(\l_s)$ is the fundamental weight corresponding to $s$, $\o \in \Omega(1) \smallsetminus Z$ and $\ts=\o\i s \o \in S_{\aff}(1)$.
\end{lem}


\begin{lem} \label{quadrac}
For $t \in J_{\aff}(1)$, we have the quadratic relation $$(T_{t}^J)^2=c_t T_t^J + q_t T_{t^2}.$$
\end{lem}
\begin{proof}
If $t \in J(1)$, then $T_t^J=T_t$ and the statement is trivial. Now we assume $t \in J_{\aff}(1) \smallsetminus J(1)$. Since $Z$ is finite and $\tW_J(1)$ is finitely generated, there exists a central element $\mu$ of $X_J^+(1)$ such that $t \mu \in \tW^+(J)(1)$. Then $T_t^J=T_{t \mu} T_{\mu}\i=T_{\mu}\i T_{\mu t}$. It remains to show $T_{\mu t}^2=c_t T_{\mu^2 t} + q_t t^2 T_{\mu^2}$.

Assume $\pi(t)=\e^{\a} s_\a \in W_{J_{\aff}}$ for some maximal short root $\a \in R_J^+$. Let $w \in W_J(1)$ and $s \in J(1)$ such that $s_\a=w s w\i$ and $\ell(s_\a)=2\ell(w)+1$. Let $\l_s \in X^+(1)$ such that $\pi(\l_s)$ is the fundamental weight corresponding to $s$. Since the quadratic relation of $T_t$ is equivalent to that of $T_{t z}$ for any $z \in Z$, we can assume $t=w \l_s s \l_s\i w\i$. Set $\l=\l_s s \l_s\i s\i \in X(1)$. We have $T_{\mu t}^J=\th_{w \mu \l w\i} T_{w s\i w\i}\i=\th_{w \mu \l w\i} T_{w\i}\i T_{s\i}\i T_w\i$. Let $w=s_1 \cdots s_n$ be a reduced expression of $w$ with each $s_i \in S_0(1)$. Then $\<s_{k-1} \cdots \s_1(\a), \a_k^\vee\>=1$ for $k=1, \dots, n$, where $\a_i$ is the simple root corresponding to $s_i$. Applying Lemma \ref{commutator} (2), for $\l' \in w \l w\i Z$ we have \[\tag{a} T_w\i \th_{w \mu\l w\i} T_{w\i}\i = T_{s_2 \cdots s_n}\i \th_{s_1\i w \mu\l w\i s_1} T_{(s_2 \cdots s_n)\i}\i = \cdots = \th_{\mu \l}.\] Similarly, we have \[\tag{b} T_{w\i}\i \th_{s \mu\l s} T_w\i=\th_{w \mu\l w\i}.\]

Case(1): $\a^\vee \notin 2Y$. Then $\l_s \in X(1)$ and $q_t=q_s$. We have \begin{align*} (T_t^J)^2 &= \th_{w \mu \l w\i} T_{w\i}\i T_{s\i}\i T_w\i \th_{w \mu \l w\i} T_{w\i}\i T_{s\i}\i T_w\i \\ &=\th_{w \mu \l w\i} T_{w\i}\i T_{s\i}\i \th_{\mu \l} T_{s\i}\i T_w\i \\ &=\th_{w \mu \l w\i} T_{w\i}\i T_{s\i}\i \th_{\mu \l_s} T_s\i T_s \th_{s \l_s\i s\i} T_{s\i}\i T_w\i \\ &= \th_{w \mu \l w\i} T_{w\i}\i \th_{s \mu \l_s s\i} (c_s+q_s s^2 T_s\i) \th_{s \l_s\i s\i} T_{s\i}\i T_w\i \\ &=c_{(w \l_s) \bullet s} \th_{w \mu^2 \l w\i} T_{w\i}\i T_{s\i}\i T_w\i \\ & \quad + q_s \th_{w \mu^2 \l w\i} T_{w\i}\i \th_{s \mu \l_s s\i} s^2 T_s\i \th_{s \l_s\i s\i} T_{s\i}\i T_w\i \\ &=c_t \th_{w \mu^2 \l w\i} T_{w\i}\i T_{s\i}\i T_w\i + q_s \th_{w \mu \l w\i} T_{w\i}\i \th_{\mu \l s \mu \l_s s \l_s\i} T_w\i \\ &= c_t \th_{w \mu^2 \l w\i} T_{w\i}\i T_{s\i}\i T_w\i+ q_s \th_{w \mu \l s \mu \l_s s \l_s\i w\i} \\ &= c_t T_{\mu^2 t} + q_t t^2 T_{\mu^2}, \end{align*} where the second equality follows from (a); the sixth one follows from Lemma \ref{commutator} (2); the seventh follows from Lemma \ref{commutator} (1).

Case(2): $\a^\vee \in 2Y$. Then $\l_s \notin X(1)$. Let $\chi, w', w'', \ts$ be as in Lemma \ref{commutator} (3). One computes that \begin{align*} (T_t^J)^2 &= \th_{w \mu \l w\i} T_{w\i}\i T_{s\i}\i \th_{\mu \l} T_{s\i}\i T_w\i \\ &=\th_{w \mu \l w\i} T_{w\i}\i T_{s\i}\i \th_\chi \th_{\chi\i \mu \l} T_{s\i}\i T_w\i \\ &=\th_{w \mu \l w\i} T_{w\i}\i T_{s\i}\i \th_\chi T_{s\i}\i \th_{s\i \chi\i \mu \l s} T_w\i \\ &=\th_{w \mu \l w\i} T_{w\i}\i T_{s\i}\i T_{w'} T_{\ts\i} T_{\ts} T_{w''} \th_{\chi\i s\i \chi\i \mu \l s} T_w\i \\ &= \th_{w \mu \l w\i} T_{w\i}\i T_{s\i}\i T_{w'} (c_{\ts} T_{\ts\i}+q_{\ts}) T_{w''} \th_{\chi\i s\i \chi\i \mu \l s} T_w\i \\ &=\th_{w \mu \l w\i} c_{(w s w') \bullet \ts} T_{w\i}\i T_{s\i}\i \th_{w' \ts\i w'' \chi\i s\i \chi\i \mu \l s} T_w\i \\ & \quad + q_{\ts} \th_{w \mu \l w\i} T_{w\i}\i \th_{s \mu \l s} T_w\i \\ &= \th_{w \mu \l w\i} c_{(w s w') \bullet \ts} T_{w\i}\i T_{s\i}\i \th_\mu T_w\i + q_{\ts} \th_{w \mu \l s \mu \l s w\i} \\ &=\th_{w \mu^2 \l w\i} c_{(w s w') \bullet \ts} T_{w\i}\i T_{s\i}\i T_w\i + q_{\ts} t^2 T_{\mu^2},\end{align*} where the first equality follows from (a); the third one follows from \ref{commutator} (1); the fourth and the sixth follow from \ref{commutator} (3); the seventh follows from (b). Note that $q_{\ts}=q_t $ since $\pi(\ts)$ and $\pi(t)$ are conjugate under $\tW$. It remains to check $$w \mu^2 \l s w' c_{\ts} {w'}\i s\i w\i=c_t w \mu^2 \l w\i,$$ that is $$\l s w' c_{\ts} {w'}\i s\i = c_{\l_s \bullet s} \l_s s \l_s\i s\i= \l_s s \l_s\i c_{\l_s \bullet s} s\i,$$ which follows by observing that $\l s=\l_s s \l_s\i$ and $w' c_{\ts} {w'}\i=c_{\l_s \bullet s}$.
\end{proof}

\begin{lem} \label{add}
Let $x, x' \in \tW_J^+(1)$ such that $\ell_J(x' x')=\ell_J(x)+\ell_J(x')$. Then $\ell(x x')=\ell(x)+\ell(x')$. As a consequence, we have $T_y^J T_{y'}^J=T_{y y'}^J$ for $y, y' \in \tW_J(1)$ such that $\ell_J(y y')=\ell_J(y)+\ell_J(y')$.
\end{lem}
\begin{proof}
Write $x=\l u$ and $x'=\l' u'$, where $\l, \l' \in X(1)$ and $u, u' \in W_0(1)$. Since $x, x', x x' \in \tW_J^+(1)$, one computes that \begin{align*} \ell(x x')&=\ell_J(x x') + \sum_{\a \in R^+ \smallsetminus R_J} |\<\l+ \pi(u)(\l'), \a^\vee\>| \\ &=\ell_J(x) +\ell_J(x') + \sum_{a \in R^+ \smallsetminus R_J} \<\l, \a^\vee\> + \sum_{a \in R^+ \smallsetminus R_J} \<\pi(u)(\l'), \a^\vee\> \\ &=\ell_J(x) + \sum_{a \in R^+ \smallsetminus R_J} \<\l, \a^\vee\>+ \ell_J(x') + \sum_{a \in R^+ \smallsetminus R_J} \<\l', \a^\vee\> \\ &=\ell(x)+\ell(x')\end{align*} as desired.
\end{proof}

\subsection{Proof of Theorem \ref{para}}
(1) Let $\tch_J(q^J, c^J)$ be the generic pro-$p$ Hecke algebra associated to $\tW_J(1)$, where $q^J_t=q_t$ and $c^J_t=c_t$ for $t \in T(1) \cap \tW_J(1)$. Denote by $(T_{w, J})_{w \in \tW_J(1)}$ its Iwahori-Matsumoto basis. Combining Lemma \ref{quadrac} and Lemma \ref{add}, we see that there exists a surjective algebra homomorphism from $\tch_J(q^J, c^J)$ to $\tch_J$ sending $T_{w, J}$ to $T_w^J$ for $w \in \tW_J(1)$. It is easy to see $(T_w^J)_{w \in \tW_J(1)}$ is linear independent. Hence the homomorphism is an isomorphism.

(2) Let $x, x' \in \tW_J^+(1)$. We have to show $T_x^J T_{x'}^J \in \tch_J^+$. We argue by induction on $\ell(x')$. If $\ell(x')=0$, $T_x T_{x'}=T_{x x'} \in \tch_J^+$ by Lemma \ref{add}. Assume $T_x T_{x''} \in \tch_J^+$ for any $x, x'' \in \tW_J^+(1)$ with $\ell(x'') < \ell(x')$. Again by Lemma \ref{add}, it remains to consider the case $\ell(x x') < \ell(x) + \ell(x')$. Let $x'=\o s_1 \cdots s_n$ be a reduced expression with respect to $\ell_J$, where $\o \in \Omega_J(1)$ and $s_i \in J_{\aff}(1)$ for $1 \le i \le n$. By assumption, there exists $1 \le m \le n$ such that $\ell_J(x' \o s_1 \cdots s_{m-1})=\ell_J(x)+\ell_J(\o s_1 \cdots s_{m-1})$ and $x \o s_1 \cdots s_m <_J x \o s_1 \cdots s_{m-1}$. By the exchange condition for the Coxeter group $W_{J_{\aff}}$, $x \o s_1 \cdots s_m= y \o s_1 \cdots s_{m-1}$ for some $y <_J x$ such that $\ell_J(y \o s_1 \cdots s_{m-1})=\ell_J(y)+ \ell_J(\o s_1 \cdots s_{m-1})$. So $y \in \tW_J^+(1)$ by Lemma \ref{alcove}.

By part (1) and Lemma \ref{quadrac}, one computes that \begin{align*} T_x^J T_{x'}^J &=T_x^J T_{\o s_1 \cdots s_{m-1}}^J T_{s_m}^J T_{s_{m+1} \dots s_n}^J \\ &=T_{x \o s_1 \cdots s_{m-1}}^J T_{s_m}^J T_{s_{m+1} \cdots s_n}^J \\ &=T_x^J T_{\o s_1 \cdots s_{m-1}}^J c_{s_m} T_{s_{m+1} \dots s_n}^J + q_{s_m} T_y^J T_{\o s_1 \cdots s_{m-1}}^J T_{s_{m+1} \dots s_n}^J. \end{align*} Note that $T_{\o s_1 \cdots s_{m-1}}^J T_{s_{m+1} \dots s_n}^J$ is a linear combination of $T_{x''}^J$ such that $x'' <_J x'$.  Again by Lemma \ref{alcove}, we have $x'' \in \tW_J^+(1)$. Now the statement follows by induction hypothesis.

\section{Standard representatives}\label{standard}

By Theorem \ref{cocenter}, the cocenter $\overline{\tch}$ of an affine pro-$p$ Hecke algebra $\tch$ is spanned by the image of $T_\tw$ for $\tw \in \tW(1)_{\min}$. In this section, we will compute the trace of $T_\tw$ for $\tw \in \tW(1)_{\min}$ using certain elements in the parabolic subalgebra $\tch_J^+$ of $\tch$.

\subsection{} Let $n_0=\sharp W_0$. For any $\tw \in \tW(1)$, $\tw^{n_0}=\l$ for some $\l \in X(1)$. Let $\nu_\tw=\l/n_0 \in X_\QQ$ and $\bar \nu_\tw \in X_\QQ^+$ be the unique dominant element in the $W_0$-orbit of $\nu_w$. It is easy to see that the map $\tW \to V, \tw \mapsto \bar \nu_\tw$ is constant on each conjugacy class of $\tW$. 

We say that an element $\tw \in \tW(1)$ is straight if $\ell(\tw^n)=n \ell(\tw)$ for any $n \in \NN$. By \cite[Lemma 1.1]{He00}, $\tw$ is straight if and only if $\ell(\tw)=\<\bar \nu_\tw, 2 \rho^\vee\>$, where $\rho$ is the half sum of positive coroots. A conjugacy class that contains a straight element is called a straight conjugacy class.

It is proved in \cite[Proposition 2.8]{HN14} that for each cyclic-shift class in $\tW(1)_{\min}$, we have some nice representatives.

\begin{prop}\label{rep}
For any $\tw \in \tW(1)_{\min}$, there exists a subset $K' \subseteq S_{\aff}$ with $W_{K'}$ finite, a straight element $y \in {}^{K'} \tW {}^{K'}(1)$ with $y K'(1) y \i=K'(1)$, and an element $w \in W_{K'}(1)$ such that $\tw \tilde \approx w y$. Here $W_{K'} \subseteq W_{\aff}$ denotes the subgroup generated by reflections of $K'$.
\end{prop}

\subsection{} In the situation of Proposition \ref{rep}, we call $w y$ a standard representative of the cyclic-shift class of $\tw$. By \cite[Proposition 2.2]{He00}, $\bar \nu_{\tw}=\bar \nu_{w y}=\bar \nu_y$. The expression of standard representative relates each conjugacy class of $\tW$ with a straight conjugacy class. It plays an important role in the study of combinatorial properties of conjugacy classes of affine Weyl groups \cite{HN14}, $\s$-conjugacy classes of $p$-adic groups \cite{He99} and representations of affine Hecke algebras with nonzero parameters \cite{CH}.

However, for a given cyclic-shift class in $\tW_{\min}(1)$, the standard representatives are in general, not unique. This leads to some difficulty in understanding the cyclic-shift classes in $\tW_{\min}(1)$ and their relations to the representations of $\tch$.

To overcome the difficulty, we use the notion of standard quadruples.

We say that $(J, \ux, \G, C)$ is a {\it standard quadruple} of $\tW$ if
\begin{itemize}

\item $J \subset S_0$;

\item $\ux \in \Omega_J$ and $\<\nu_{\underline{x}}, \a^\vee\>>0$ for all $\a \in R^+ \smallsetminus R_J$;

\item $\G \subseteq J_{\aff}$ such that $W_\G$ finite and the conjugation action of $\ux$ stabilizes $\G$.

\item $C$ is an elliptic $\Ad(\ux)$-twisted conjugacy class of $W_\G$.
\end{itemize}

We say that $(J, \ux, \G, C)$ and $(J', \ux', \G', C')$ are {\it strongly equivalent} if $J=J'$ and there exists $\underline{\o} \in \Omega_J$ such that $(\ux', \G', C')$ is obtained from $(\ux, \G, C)$ by conjugation by $\underline{\o}$. It is proved in \cite[Proposition 3.23]{HeX} that

\begin{prop} \label{quadruple}
The map $(J, \ux, \G, C) \mapsto [C_{\min} \ux]$ induces a bijection between the strongly equivalence classes of standard quadruples and cyclic-shift classes in $\tW_{\min}$.
\end{prop}

\subsection{}\label{pair}

Let $w y$ be a standard representative as in Proposition \ref{rep}. Let $J=J_{\bar \nu_y} \subseteq S_0$ and $K=\cup_i \pi(y)^i \supp(w) \pi(y)^{-i} \subseteq S_{\aff}$. By loc. cit., $W_E$ is finite. Let $h \in {}^J W_0(1)$ such that $h(\nu_y)=\bar \nu_y$. Set $x=h y h\i$ and $\G=\pi{h} K \pi{h}\i \subseteq J_{\aff}$. Denote by $C$ the $\Ad(\pi(x))$-twisted conjugacy class of $W_\G$. By construction, $C$ is elliptic. One checks that $(J, \pi(x), \G, C)$ is the standard quadruple corresponding to the cyclic-shift class of $\pi(wy)$ in Proposition \ref{quadruple}. In this paper, we are mainly interested in the pair $(J, \G)$ associated to $w y$. We call it the {\it associated standard pair}. Note that $\G \subset J_{\aff}$ with $W_{\G}$ finite.

\smallskip

Now we state the main result of this section.
\begin{prop} \label{power}
Let $w, y, h, J, K$ be as in \S\ref{pair}. Then for $n \gg 0$, $$T_{w y}^n \equiv (T_{h w y h\i}^J)^n \mod [\tch, \tch].$$
\end{prop}

\begin{rmk}
Note that $(T_{h w y h\i}^J)^n \in \tch_J^+ \subseteq \tch$.
\end{rmk}

The following is a variation of the length formula in \cite{IM}.

\begin{lem}
For $w \in W_0(1)$ and $\a \in R$, set $$\d_w(\a)=\begin{cases} 0, & \text{ if } w \a \in R^+; \\ 1, & \text{ if } w \a \in R^-. \end{cases}$$ Then for any $x, y \in W_0(1)$ and $\mu \in X(1)$, we have that $$\ell(x t^\mu y)=\sum_{\a \in R^+} |\<\mu, \a^\vee\>+\d_x(\a)-\d_{y \i}(\a)|.$$
\end{lem}


\begin{prop}\label{u-x}
Let $(J, \G)$ be a standard pair. Let $x \in \Omega_J(1)$ such that $\nu_x \in X_\QQ^+(J)$. Then for any $u \in W_\G(1)$ we have

(1) for $n \gg 0$ and $h \in {}^J W_0(1)$, $\ell(h \i u x^n  h)=\ell(u x^n)$.

(2) for $n \gg 0$, $\ell(u x^{n+n_0})=\ell(u x^n)+\ell(x^{n_0})$, where $n_0=\sharp W_0$.
\end{prop}

\begin{proof}
We have $u x^n=\l w$ for some $\l \in X(1)$ and $w \in W_J(1)$. Since $\<\nu_x, \a^\vee\>>0$ for any $\a \in R^+ \smallsetminus R_J$, we have $\<\l, \a^\vee\>>0$ for any $\a \in R^+ \smallsetminus R^+_J$ as $n \gg 0$.

Notice that for $\a \in R_J$, $\d_{h \i(\a)}=\d_{\a}$. Now \begin{align*} & \ell(h \i u x^n h)=\sum_{\a \in R^+} |\<\l, \a^\vee\>+\d_{h \i}(\a)-\d_{h \i w \i}(\a)| \\ &=\sum_{\a \in R^+_J} |\<\l, \a^\vee\>-\d_{w \i}(\a)|+\sum_{\a \in R^+ \smallsetminus R^+_J} |\<\l, \a^\vee\>+\d_{h \i}(\a)-\d_{h \i w \i}(\a)| \\ &=\sum_{\a \in R^+_J} |\<\l, \a^\vee\>-\d_{w \i}(\a)|+\sum_{\a \in R^+ \smallsetminus R^+_J} \bigl(\<\l, \a^\vee\>+\d_{h \i}(\a)-\d_{h \i w \i}(\a) \bigr) \\ &=\sum_{\a \in R^+_J} |\<\l, \a^\vee\>-\d_{w \i}(\a)|+\sum_{\a \in R^+ \smallsetminus R^+_J} \<\l, \a^\vee\>+\sharp\{\a \in R^+ \smallsetminus R^+_J, h \i(\a) \in R^-\} \\ & \qquad-\sharp\{\a \in R^+\smallsetminus R^+_J, h \i w \i(\a) \in R^-\} \\ &=\sum_{\a \in R^+_J} |\<\l, \a^\vee\>-\d_{w \i}(\a)|+\sum_{\a \in R^+ \smallsetminus R^+_J} \<\l, \a^\vee\>+\ell(h)-\ell(h) \\ &=\sum_{\a \in R^+_J} |\<\l, \a^\vee\>-\d_{w \i}(\a)|+\sum_{\a \in R^+ \smallsetminus R^+_J} \<\l, \a^\vee\>.
\end{align*}

This proves part (1).

For part (2), \begin{align*} & \ell(u x^{n+n_0})=\sum_{\a \in R^+} |\<\l+n_0 \nu_x, \a^\vee\>-\d_{w \i}(\a)| \\ &=\sum_{\a \in R^+_J} |\<\l+n_0 \nu_x, \a^\vee\>-\d_{w \i}(\a)|+\sum_{\a \in R^+\smallsetminus R^+_J} \<\l+n_0 \nu_x, \a^\vee\> \\ &=\sum_{\a \in R^+_J} |\<\l, \a^\vee\>-\d_{w \i}(\a)|+\sum_{\a \in R^+\smallsetminus R^+_J} \<\l, \a^\vee\>+\sum_{\a \in R^+\smallsetminus R^+_J} \<n_0 \nu_x, \a^\vee\> \\ &=\ell(u x^n)+\ell(x^{n_0}).\qedhere
\end{align*}
\end{proof}

As a consequence, we have
\begin{cor}\label{xG}
Let $w, y, h, J$ be as in \S \ref{pair}.  Then for $n \gg 0$ we have $w y^n \tilde \approx h w y^n h\i$.
\end{cor}
\begin{proof}
Let $x=h y h\i$ and $u \in h w h\i$. Suppose that $h=s_1 \cdots s_k$ for $s_1, \cdots, s_k \in S_0(1)$. Set $h_i=s_1 \cdots s_i$ for $1 \le i \le k$. Then $h_i \in {}^J W_0(1)$. By Proposition \ref{u-x} (1), $\ell(h_i \i u x^n h_i)=\ell(h_{i+1} \i u x^n h_{i+1})$ for $0 \le i \le k-1$. Hence $h_i \i u x^n h_i \tilde \approx h_{i+1}\i u x^n h_{i+1}$ for $0 \le i \le k-1$. Therefore $u x^n \tilde \approx h \i u x^n h=w y^n$.
\end{proof}

\subsection{Proof of Proposition \ref{power}}
Assume $$T_w T_{y w y\i} \cdots T_{y^{n-1} w y^{1-n}} T_{y^n} =\sum_{w' \in W_K(1)} a_{w'} T_{w'}$$ with $a_{w'} \in \kk$. Let $\ch_K \subseteq \tch$ (resp. $\ch_{J, \G} \subseteq \tch_J$) be the subalgebra generated by $T_w$ for $w \in W_K(1)$ (resp. by $T_w^J$ for $w \in W_\G(1)$). By Lemma \ref{quadrac}, the map $T_{w'} \mapsto T_{h w' h\i}^J$ gives an algebra isomorphism between $\ch_K$ and $\ch_{J, \G}$. Thus $$T_{h w h\i}^J T_{h y w y\i h\i}^J \cdots T_{h y^{n-1} w y^{1-n} h\i}^J=\sum_{w' \in W_K(1)} a_{w'} T_{h w' h\i}^J.$$ Now one computes that \begin{align*} T_{w y}^n &= T_w T_{y w y\i} \cdots T_{y^{n-1} w y^{1-n}} T_{y^n} =(\sum_{w' \in W_K(1)} a_{w'} T_{w'}) T_{y^n} \\ &= \sum_{w' \in W_K(1)} a_{w'} T_{w' y^n} \equiv \sum_{w' \in W_K(1)} a_{w'} T_{h w' y^n h\i} \mod [\tch, \tch]. \end{align*} Moreover \begin{align*} \sum_{w' \in W_K(1)} a_{w'} T_{h w' y^n h\i}  &= \sum_{w' \in W_K(1)} a_{w'} T_{h w' y^n h\i}^J \\ &=(\sum_{w' \in W_K(1)} a_{w'} T_{h w' h\i}^J) T_{h y^n h\i}^J \\ &= T_{h w h\i}^J T_{h y w y\i h\i}^J \cdots T_{h y^{n-1} w y^{1-n} h\i}^J T_{h y^n h\i}^J \\ &=(T_{h w y h\i}^J)^n. \qedhere \end{align*}

\section{Some character formulas}\label{6}

\subsection{} Let $M \in R(\tch)_\kk$. For any $J \subseteq S_0$, we set $M_J=\cap_{\l \in X^+(J)(1)} T_\l M$. Since $M$ is a finite dimensional, there exists $\mu \in X^+(J)(1)$ such that $M_J=T_\mu M$. Moreover, since the action of $T_\l$ on $M_J$ is invertible for any $\l \in X^+(J)(1)$, we may regard $M_J$ as an $\tch_J$-module. For $\G \subset J_{\aff}$, let $$\Omega_J(\G)=\{\t \in \Omega_J(1); \pi(\t) \G \pi(\t)\i=\G\}$$ and $M_{J, \G}=T^J_{w_\G} M_J$, where $w_\G \in \tW_J(1)$ such that $\pi(w_\G)$ is the longest element of $W_\G$. Then $M_{J, \G}$ is an $\Omega_J(\G)(1)$-module.

Let $\ch_{J, \G} \subseteq \ch_J$ be the subalgebra spanned by $T_u^J$ with $u \in W_\G(1)$. By \S\ref{3-1}, each irreducible $\ch_{J, \G} \rtimes \Omega_J(\G)(1)$ is of the form $I(\Xi, V)$ for some permissible pair $(\Xi, V)$ with respect to $(W_\G(1), \Omega_J(\G)(1))$. Let $u \in \tW_J(1)$ such that $\supp^J(u)=\supp^J(w_\G)=\G$. One checks directly that $T_u^J I(\Xi, V)=T_{w_\G}^J I(\Xi, V)$. In particular, we have $M_{J, \G}=T_u^J M_J \in R(\ch_{J, \G} \rtimes \Omega_J(\G)(1))$.




\subsection{}\label{aleph} Let $\aleph=\{(J, \G); J \subset S_0, \G \subset J_{\aff}\}$ and let $\aleph^*=\{(J, \G) \in \aleph; \sharp W_\G < +\infty\}$ be the set of standard pairs. We define an equivalence relation $\sim$ and a partial order $<$ on $\aleph$ as follows. Let $(J, \G), (J', \G') \in \aleph$. We say $(J, \G) \sim (J', \G')$ if $J=J'$ and $\G'=\pi(\t) \G \pi(\t)\i$ for some $\t \in \Omega_J(1)$. We say that $(J, \G)<(J', \G')$ if either $J \subsetneq J'$ or $J=J'$ and $\G \supsetneq \pi(\t) \G' \pi(\t)\i$ for some $\t \in \Omega_J(1)$.

For $(J, \G) \in \aleph$, denote by $\cp(J, \G)$ the set of permissible pairs $(\Xi, V)$ with respect to $(W_\G(1), \Omega_J(\G)(1))$ such that $\Xi(T_w^J) \neq 0$ for each $w \in W_\G(1)$. Let $(\Xi, V) \in \cp(J, \G)$ and let $I(\Xi, V) $ be the $\ch_{J, \G} \rtimes (\Omega_J(\G)(1))$-module constructed as in Section 3. Set $\tch_J(\G)=\ch_J \rtimes (\Omega_J(\G)(1))$. We denote by $I(\Xi, V)_0$ the extension of $I(\Xi, V)$ by zero as a module of $\tch_J(\G)$ by requiring that $T_w I(\Xi, V)_0=0$ if $\pi(w) \notin W_\G \rtimes \Omega_J(\G)$. Define $$\pi_{J, \G, \Xi, V}=\tch \otimes_{\tch_J^+} (\tch_J \otimes_{\tch_J(\G)} I(\Xi, V)_0).$$ It is easy to see that $\pi_{J, \G, \Xi, V}$ and $\pi_{J, {}^\g \G, {}^\g \Xi, {}^\g V}$ are isomorphic as $\tch$-modules for $\g \in \Omega_J(1)$.

\begin{thm} \label{char}
Let $(J, \G), (J', \G') \in \aleph^*$ and $(\Xi, V) \in \cp (J, \G)$. Then \begin{align*} (\pi_{J, \G, \Xi, V})_{J'}=\begin{cases} \oplus_{\g \in \Omega_J(1) / \Omega_J(\G)(1)} {}^\g I(\Xi, V), & \text{ if } J=J'; \\ 0, & \text{ if }  J \nsubseteq J'. \end{cases} \end{align*} Moreover, \begin{align*} (\pi_{J, \G, \Xi, V})_{J', \G'}=\begin{cases}I(\Xi, V), & \text{ if } (J, \G)=(J', \G'); \\ 0, & \text{ if } (J, \G) \nleqslant (J', \G'). \end{cases}\end{align*}
\end{thm}
\begin{proof}
Let $\l \in X^+(J')(1)$ such that $T_\l \pi_{J, \G, \Xi, V}=(\pi_{J, \G, \Xi, V})_{J'}$. We may replace $\l$ with some appropriate power of itself so that $T_{\g\i \l \g}^J \in \tch_J(\G)$ for any $\g \in \Omega_J(1)$. Let $$M=\oplus_{\g \in \Omega_J(1)/\Omega_J(\G)(1)} T_\g^J \otimes I(\Xi, V).$$ By \cite[Proposition 5.2]{O10}, we have $$\pi_{J, \G, \Xi, V} =\oplus_{d \in W_0^J(1)} {}^\iota T_d \otimes M.$$ For $s \in S_0(1)$, $$T_\l {}^\iota T_s=\begin{cases} {}^\iota T_s T_{s\i \l s}, & \text{ if } s \in J'(1); \\ 0, & \text{ otherwise.} \end{cases}$$ Thus for $d \in W_0^J(1)$, \[\tag{a} T_\l {}^\iota T_d=\begin{cases} {}^\iota T_d T_{d\i \l d}, & \text{ if } d \in W_{J'}(1); \\ 0, & \text{ otherwise,} \end{cases}\] where $d \i \l d \in \l Z$ since $\l \in X(J')(1)$.

Assume $J \nsubseteq J'$. We show that $(\pi_{J, \G, \Xi, V})_{J'}=T_\l \pi_{J, \G, \Xi, V}=0$. By (a), it suffices to show $T_\l M=0$. One checks that $$T_\l (T_\g^J \otimes I(\Xi, V)_0)=T^J_\l (T_\g^J \otimes I(\Xi, V)_0)=T_\g^J \otimes (T^J_{\g\i \l \g} I(\Xi, V)_0).$$

Since $J \not \subset J'$, there exists $\b \in R_J$ such that $\<\nu_\l, \b^\vee\> \neq 0$. We have $\g \i \l \g=w \l w\i$ for some $w \in W_J(1)$. Thus $\<\nu_{\g \i \l \t}, \pi(w)\i(\b^\vee)\> \neq 0$. Hence $\g\i \l \g \notin W_\G \rtimes \Omega_J(\G)(1)$ and $T^J_{\g\i \l \g} I(\Xi, V)_0=0$.

Assume $J=J'$. By (a), we have $$(\pi_{J, \G, \Xi, V})_{J'} =T_\l \pi_{J, \G, \Xi, V}= M = \oplus_{\g \in \Omega_J(1) / \Omega_J(\G)(1)} I({}^\g \Xi, {}^\g V)_0.$$ Let $u \in W_{\G'}$ with $\supp^J(u)=\G'$. If $(J, \G) \nleqslant (J, \G')$, that is, $\G' \nsubseteq {}^\g \G $ for each $\g \in \Omega_J(1)$. So ${}^\g \Xi(T_u)=0$ and hence $$(\pi_{J, \G, \Xi, V})_{J, \G}=T_u M = \oplus_{\g \in \Omega_J(1) / \Omega_J(\G)(1)}T_u I({}^\g \Xi, {}^\g V)=0.$$ If $\G=\G'$, once checks similarly $T_u I({}^\g \Xi, {}^\g V) \neq 0$ if and only if $\g \in \Omega_J(\G)(1)$. So $(\pi_{J, \G, \Xi, V})_{J, \G}=I(\Xi, V)$ and proof is finished.
\end{proof}

\smallskip

Now we state the main result of this section.

\begin{thm}\label{char-1}
Let $(J, \G) \in \aleph^*$ and $(\Xi, V) \in \cp(J, \G)$. Let $w y$ be a standard representative and $(J', \G')$ be the standard pair assoicated to it. Let $h \in {}^J W_0(1)$ with $h(\nu_y)=\bar \nu_y$. Then \begin{align*} & Tr(T_{w y}, \pi_{J, \G, \Xi, V}) \\& =\begin{cases}
\sum_{\g \in \Omega_J(\G)(1)/\Omega_J(\G, \Xi)(1)} {}^\g \Xi(T_{h w h \i}) Tr(h y h \i, {}^\g V), & \text{ if } (J, \G)=(J', \G'); \\
0, & \text{ if } (J, \G) \nleqslant (J', \G').
\end{cases}\end{align*} Here $\Omega_J(\G, \Xi)(1)$ is the stabilizer of $\Xi$ in $\Omega_J(\G)(1)$.
\end{thm}

\begin{lem} \label{p-trace}
Let $(J, \G)$ be a standard pair. Let $x, x' \in \Omega_J(\G)(1)$ such that $\nu_x \in X^+(J)_\QQ$, and let $u \in W_\G(1)$ with $\supp^J (u)=\G$. Let $M \in R(\tch)$. Then for $n \gg 0$, we have $$Tr(T_{u x' x^n}, M)=Tr(T_{u x' x^n}^J, M_{J, \G}).$$
\end{lem}
\begin{proof}
Let $\mu \in X^+(J)(1)$ such that $M_J=T_\mu M$. Notice that $n_0 \nu_x \in X^+(J)$, where $n_0=\sharp W_0$. There exists $m \in \NN$ such that $x^{m n_0} \mu\i \in X^+(J)$. By Proposition \ref{u-x} (2), for $n \gg 0$, $\ell(ux' x^{n+m n_0})=\ell(ux' x^n)+\ell(x^{m n_0})=\ell(ux' x^n)+\ell(x^{m n_0} \mu\i)+\ell(\mu)$ and $$T_{ux' x^{n+m n_0}}=T_{ux' x^n} T_{x^{m n_0} \mu\i} T_\mu.$$ Moreover, for $n \gg 0$, $ux' x^{n+m n_0} \in \tW_J^+(1)$ and $T_{ux' x^{n+m n_0}}=T_{ux' x^{n+m n_0}}^J$. Since $0 \to \ker(T_\mu: M \to M) \to M \to M_J \to 0$, we have \begin{align*} Tr(T_{ux' x^{n+m n_0}}, M) &=Tr(T_{ux' x^n} T_{x^{m n_0} \mu\i} T_\mu, M)=Tr(T_{ux' x^{n+m n_0}}, M_J) \\ &=Tr(T^J_{ux' x^{n+m n_0}}, M_J).\end{align*}

Notice that $T^J_{ux' x^{n+m n_0}}=T^J_{ux'} (T^J_x)^{n+m n_0}=T_{x'}^J (T^J_x)^{n+m n_0} T^J_{u'}$ for some $u'$ with $\supp^J(u')=\G$. Since $0 \to \ker(T_{u'}^J: M_J \to M_J) \to M_J \to M_{J, \G} \to 0,$ we have $$Tr(T^J_{u x^{n+m n_0}}, M_J)=Tr(T_{x'}^J (T^J_x)^{n+m n_0} T^J_{u'}, M_J)=Tr(T^J_{ux' x^{n+m n_0}}, M_{J, \G})$$ as desired.
\end{proof}

\subsection{Proof of Theorem \ref{char-1}}
Let $x=h y h\i \in \Omega_{J'}(1)$ and $u=h w h\i \in W_{\G'}(1)$. Let $n \in \NN$. Then $(T_{h w y h\i}^J)^n=(T_{u x}^J)^n$ is a linear combination of $T_{u' x^n}$ with $u' \in W_{\G'}(1)$. Thus, for $M \in R(\tch)$ and $n \gg 0$ we that \begin{align*} Tr(T_{w y}^n, M) &= Tr((T_{h w y h\i}^J)^n, M)= Tr((T_{u x}^{J'})^n, M_{J', \G'}), \end{align*} where the first equality follows from Proposition \ref{power}, and the second one follows from Lemma \ref{p-trace}. Thus by Theorem \ref{char},
\begin{align*} Tr(T^n_{w y}, \pi_{J, \G, \Xi, V)}) &=Tr(T^{J'}_{u x}, (\pi_{J, \G, \Xi, V})_{J', \G'}) \\ &=\begin{cases} Tr(T^J_{u x}, I(\Xi, V)), & \text{ if } (J, \G)=(J', \G'); \\ 0, & \text{ if } (J, \G) \nleqslant (J', \G'). \end{cases}\end{align*}
Now the statement follows from Lemma \ref{ftrace}.

\section{Representations of $\tch$}

We first give a basis of $R(\tch)_\kk$.

\begin{thm}\label{main-7}
The set $$\{\pi_{J, \G, \Xi, V}; (J, \G) \in \aleph^*/\sim, (\Xi, V) \in \cp(J, \G) / \sim\}$$ is a $\kk$-basis of $R(\tch)_\kk$.
\end{thm}

\begin{lem} \label{central}
Let $A$ be a $\kk$-algebra. Let $\t \in A$ and $\z \in R(A)$. Assume there exists an invertible central element $\mu \in A$ such that $Tr(\t \mu^n, \z)=0$ for $n \gg 0$. Then $Tr(\t, \z)=0$.
\end{lem}
\begin{proof} Assume $\z=\sum_V a_V V$, where $a_V \in \kk$ and $V$ ranges over simple modules. Since $\mu$ is central, $\mu$ acts on $V$ by a scalar $\chi_{V, \mu} \in \kk^\times$. By assumption, for $n \gg 0$ we have $$0=Tr(\t \mu^n , \z)=\sum_V a_V Tr(\t \mu^n , V)=\sum_V \chi_{V, \mu}^n a_V Tr(\t, V).$$ Due to the non-vanishing of Vandermonde determinant, for each $f \in \kk^\times$, we have $$\sum_{V, \chi_{V, \mu}=f} a_V Tr(\t, V)=0.$$ So $Tr(\t, \z)=\sum_V Tr(\t, V)=0$.
\end{proof}

\subsection{}\label{indept} First we show that $$\{\pi_{J, \G, \Xi, V}; (J, \G) \in \aleph^*/\sim, (\Xi, V) \in \cp(J, \G) / \sim\}$$ is linearly independent in $R(\tch)_\kk$.

Suppose $\sum_{J, \G, \Xi, V} a_{J, \G, \Xi, V} \pi_{J, \G, \Xi, V}=0$ with $a_{J, \G, \Xi, V} \in \kk$.

Let $(J_1, \G_1) \in \aleph^*/\sim$ be a minimal element such that $a_{J_1, \G_1, \Xi_1, V_1} \neq 0$ for some $(\Xi_1, V_1)$. Since $Z$ is finite and $\Omega$ is finitely generated, there exists a central element $\mu$ of $\Omega_{J_1}(\G_1)(1)$ with $\mu \in X^+(J_1)(1)$. Let $u \in W_{\G_1}(1)$ with $\supp^{J_1}(u)=\G_1$ and $x \in \Omega_J(\G)(1)$. Combining Lemma \ref{p-trace} with Theorem \ref{char}, we deduce that for $n \gg 0$ \begin{align*} 0 &=\sum_{J, \G, \Xi, V} a_{J, \G, \Xi, V} Tr(T_{u x \mu^n}, \pi_{J, \G, \Xi, V}) \\ &=\sum_{(\Xi, V) \in \cp(J_1, \G_1) / \sim} a_{J_1, \G_1, \Xi, V} Tr(T_{u x \mu^n}, \pi_{J_1, \G_1, \Xi, V}) \\ &= \sum_{(\Xi, V) \in \cp(J_1, \G_1) / \sim} a_{J_1, \G_1, \Xi, V} Tr(T_{u x \mu^n}^J, (\pi_{J_1, \G_1, \Xi, V})_{J_1, \G_1}) \\ &=\sum_{(\Xi, V) \in \cp(J_1, \G_1) / \sim} a_{J_1, \G_1, \Xi, V} Tr(T_{u x \mu^n}^J, I(\Xi, V)) \end{align*} By Lemma \ref{central}, we have that $$\sum_{(\Xi, V) \in \cp(J_1, \G_1) / \sim} a_{J_1, \G_1, \Xi, V} Tr(T_{u x}^J, I(\Xi, V))=0.$$ Thanks to Proposition \ref{seperate} (where we take $S=\G_1$ and $\Omega=\Omega_{J_1}(\G_1)$),  $a_{J_1, \G_1, \Xi, V}=0$ for every $(\Xi, V)$. That is a contradiction.

\subsection{} Next we show that $(\pi_{J, \G, \Xi, V})_{(J, \G) \in \aleph^*, (\Xi, V) \in \cp(J, \G)}$ spans $R(\tilde \ch)$.

For any $M \in R(\tilde \ch)$, let $\aleph^*(M)$ be the set of pairs $(J, \G)$ in $\aleph^*/\sim$ which is associated to some standard representative $\tw \in \tW(1)_{\min}$ such that $Tr(T_{\tw}, M) \neq 0$.

Fix a total order on $\aleph^*$ that is compatible with the partial order given in \S\ref{aleph}. We argue by induction on the minimal element in $\aleph^*(M)$.

If $\aleph^*(M)=\emptyset$, then $Tr(T_{w}, M)=0$ for all $w \in \tW(1)_{\min}$. By Theorem \ref{cocenter}, $Tr(h, M)=0$ for all $h \in \tch$. Hence $M=0$.

Now suppose that $\aleph^*(M) \neq \emptyset$. Let $(J, \G)$ be the minimal element in $\aleph^*(M)$. We regard $M_J$ as a virtual $\ch_{J, \G} \rtimes \Omega_J(\G)(1)$-module. Then $M_J$ is linear combination of $I(\Xi, V)$, where $(\Xi, V)$ ranges over permissible pair with respect to $(W_\G(1), \Omega_J(\G)(1))$. Therefore, $$M_{J, \G}=\sum_{(\Xi, V) \in \cp(J, \G) / \sim} a_{(\Xi, V)} I(\Xi, V),$$ where $a_{(\Xi, V)} \in \kk$. We set $$U(J, \G)=\sum_{(\Xi, V) \in \cp(J, \G) / \sim} a_{(\Xi, V)} \pi_{J, \G, \Xi, V}.$$ By Theorem \ref{char}, $M_{J, \G}=U(J, \G)_{J, \G}$. By Theorem \ref{char-1}, we deduce that $Tr(T_\tw, M)=Tr(T_{\tw}, U(J, \G))$ for each standard representative $\tw \in \tW(1)_{\min}$ whose associated standard pair is equivalent to $(J, \G)$.

Set $M'=M-U(J, \G).$ By Theorem \ref{char}, for any standard representative $\tw' \in \tW(1)_{\min}$ with $Tr(T_{w'}, M') \neq 0$, its associated standard pair $(J', \G')$ is larger than $(J, \G)$ in the fixed linear order. By inductive hypothesis, $M'$ is a linear combination of $(\pi_{J, \G, \Xi, V})_{(J, \G) \in \aleph^*, (\Xi, V) \in \cp(J, \G)}$. So does $M$, as desired.

\subsection{}\label{nss} Motivated by \cite{CH}, we introduce rigid part of the cocenter. Recall that $\{T_w\}_{w \in \tW(1)_{\min}}$ spans $\overline{\tilde \ch}$.

Let $\overline{\tilde \ch}^{rig}$ be the subspace of $\overline{\tch}$ spanned by $T_w$ for $w \in \tW(1)_{\min}$ with $\nu_w$ central and let $\overline{\tilde \ch}^{nrig}$ be the subspace of $\overline{\tch}$ spanned by $T_w$ for $w \in \tW(1)_{\min}$ with $\nu_w$ non-central. We call $\overline{\tilde \ch}^{rig}$ the {\it rigid part} of the cocenter and $\overline{\tilde \ch}^{nrig}$ the {\it non-rigid part} of the cocenter.

Let $\overline{\tch}^{nss}=\overline{\tch}^{nrig}+\iota(\overline{\tch}^{nrig})$. We call $\overline{\tch}^{nss}$ the {\it non-supersingular part} of the cocenter.

\begin{lem} \label{nrig}
Let $w \in \tW(1)$ such that $\sharp W_{\supp(w)}=+\infty$, then the image of $T_w$ in $\bar \tch$ lies in $\overline{\tch}^{nrig}$.
\end{lem}
\begin{proof}
If $w \in \tW(1)_{\min}$, then $\nu_x$ is central if and only if $\sharp W_{\supp_x} < +\infty$ and the statement follows. Assume the statement holds for any $w'' \in \tW(1)$ with $\ell(w'') < \ell(w)$. Assume $w \notin \tW(1)_{\min}$. By Theorem \ref{rep}, there exists $w \in \tW(1)$ and $s \in S_{aff}(1)$ such that $w \tilde \approx w'$ and $s w' s < w'$. We have $T_w =T_{w'}=c_s T_{w' s\i} \in \bar \tch$ and $\supp(w)=\supp(w')=\supp(w' s\i)$. Note that $\ell(w's) < \ell(w')=\ell(w)$, the statement follows by induction hypothesis.
\end{proof}

Let $M \in R(\tilde \ch)_\kk$. We say $M$ is rigid if $Tr(\overline{\tch}^{nrig}, M)=0$.

\begin{prop} \label{rig}
Let $M \in R(\tch)_\kk$. Then $M$ is rigid if and only if $M$ is spanned by $(\pi_{S_0, \G, \Xi, V})_{(S_0, \G) \in \aleph^*, (\Xi, V) \in \cp(S_0, \G)}$.
\end{prop}

\begin{proof}
By Theorem \ref{char}, $\pi_{S_0, \G, \Xi, V}$ is rigid.

On the other hand, assume that $$M=\sum_{(J,\G) \in \aleph^*/\sim \text{ with } J \subsetneq S_0, (\Xi, V) \in \cp(S_0, \G) / \sim} a_{J, \G, \Xi, V} \pi_{J, \G, \Xi, V}$$ with $a_{J, \G, \Xi, V} \in \kk$. Let $(J_1, \G_1)$ be a minimal standard pair such that $a_{J_1, \G_1, \Xi_1, V_1} \neq 0$ for some $(\Xi_1, V_1)$. Using Lemma \ref{nrig} and the same argument as in \S\ref{indept}, one deduce that $a_{J_1, \G_1, \Xi_1, V_1}=0$ for all $(\Xi_1, V_1)$. That is a contradiction.
\end{proof}

We also have the following result, which will be used in the study of supersingular represenations.

\begin{lem}\label{inf-rig}
Let $M=\sum_{(S_0, \G) \in \aleph/ \sim, \ (\Xi, V) \in \cp(S_0, \G)} a_{\G, \Xi, V} \pi_{S_0, \G, \Xi, V}$ for some $a_{\G, \Xi, V} \in \kk$. Then $M$ is rigid if and only if $a_{\G, \Xi, V}=0$ unless $\sharp W_\G<\infty$.
\end{lem}

\begin{proof}
We prove by contradiction. Let $(S_0, \G_1)$ be the minimal pair in $\aleph / \sim$ such that $\sharp W_{\G_1}= + \infty$ and $a_{\G_1, \Xi_1, V_1} \neq 0$ for some $(\Xi_1, V_1) \in \cp(S_0, \G_1)$. Let $\t \in \Omega(1)$ and $w_0, w \in W_\G(1)$ with $\supp(w_0)=\G_1$. Then $T_{w_0} T_w T_\t$ is a linear combination of $T_x$ with $\supp(x) \supseteq \G_1$. Since $\sharp W_{\G_1}= + \infty$, we have $T_{w_0} T_w T_\t \in \bar \tch^{nrig}$ by Lemma \ref{nrig}. Then $$0=Tr(T_{w_0} T_w T_\t , M) =\sum_{(\Xi, V) \in \cp(S_0, \G_1) / \sim} a_{\G_1, \Xi, V} Tr(T_{w_0} T_w T_\t, \pi_{S_0, \G, \Xi, V}),$$ where the second equality follows from the observation that $T_x \pi_{S_0, \G, \Xi, V}=0$ if $(S_0, \G) \nleqslant (S_0, \supp(x))$. By the same argument as in \S\ref{indept}, we deduce that each $a_{\G_1, \Xi, V}$ vanishes. That is a contradiction.
\end{proof}

\smallskip

Now we describe the image of rigid representations of $\tch$ under the involution $\iota$.


\begin{prop} \label{iota}
Let $(S_0, \G) \in \aleph$ and let $(\Xi, V) \in \cp(S_0, \G)$. Let $\G'=\{s \in S_{\aff} \smallsetminus \G; \Xi(c_s) \neq 0\}$ and $\Xi'$ be the character of $\ch_{\G'}$ defined by $\Xi'|_Z=\Xi|_Z$ and $\Xi'(s)=\Xi(c_s)$ for $s \in \G'$. Then
$$\iota(\pi_{S_0, \G, \Xi, V}) \cong \pi_{S_0, \G', \Xi', V}.$$
\end{prop}

\begin{lem}\label{double-st}
Let  $\Upsilon$ be a character on $\ch$. Then $$\Omega(\Upsilon)(1)=\Omega(\Upsilon |_Z)(1) \cap \Omega(\G_{\Upsilon})(1),$$ where $\G_{\Upsilon}=\{s \in S_{\aff}; \Upsilon(s) \neq 0\}$.
\end{lem}

\begin{proof}
By definition, $\Omega(\Upsilon)(1) \subset \Omega(\Upsilon |_Z)(1) \cap \Omega(\G_{\Upsilon})(1)$. On the other hand, let $\g \in \Omega(\Upsilon |_Z)(1) \cap \Omega(\G_\Upsilon)(1)$. If $s \in S_{\aff}(1) \smallsetminus \G_\Upsilon(1)$, then $\g s \g\i \in S_{\aff} \smallsetminus \G_\Upsilon$ and $\Upsilon(T_{\g s \g\i})=\Upsilon(T_s)=0$. If $s \in \G_\Upsilon(1)$, then $\g s \g\i \in \G_\Upsilon(1)$ and $\Upsilon(T_{\g s \g\i})=\Upsilon(c_{\g s \g\i})=\Upsilon(\g c_s \g\i)=\Upsilon(c_s)=\Upsilon(T_s)$. Therefore, $\g \in \Omega(\Upsilon)(1)$.
\end{proof}

\subsection{Proof of Proposition \ref{iota}} We denote by $\Xi_0$ the extension of $\Xi$ on $\ch$ defined by $\Xi_0(T_w)=0$ if $w \notin W_\G(1)$. Note that for any $s \in S_{\aff}(1)$, $\iota(T_s)=-T_s+c_s$. By Lemma \ref{double-st}, \begin{gather*}
\Omega(\Xi_0)(1)=\Omega(\Xi|_Z)(1) \cap \Omega(\G)(1), \\ \Omega(\iota (\Xi_0))(1)=\Omega(\Xi|_Z)(1) \cap \Omega(\G')(1).
\end{gather*}
Let $\g \in \Omega(\Xi |_Z)(1)$. Then $\g$ preserves $\{s \in S_{\aff}; \Xi(c_s) \neq 0\}$. Hence $\g \in \Omega(\G)(1)$ if and only $\g \in \Omega(\G')(1)$. Thus $\Omega(\Xi_0)(1)=\Omega(\iota (\Xi_0))(1)$.

We have \begin{align*} \pi_{S_0, \G, \Xi, V}&=\tch \otimes_{\tch(\G)} I(\Xi, V)_0 \cong \tch \otimes_{\tch(\G)} (\tch(\G) \otimes_{\ch \rtimes \Omega(\Xi_0)(1)} (\Xi_0 \otimes V)) \\ & \cong \tch \otimes_{\ch \rtimes \Omega(\Xi_0)(1)} (\Xi_0 \otimes V). \end{align*}

Note that $\iota(\Xi_0)=\Xi'_0$. Then \begin{align*} \iota(\pi_{S_0, \G, \Xi, V}) &\cong \iota (\tch \otimes_{\ch \rtimes \Omega(\Xi_0)(1)} (\Xi_0 \otimes V)) \\ &\cong \tch \otimes_{\ch \rtimes \Omega(\iota(\Xi_0))(1)} (\iota (\Xi_0) \otimes V) \\ &\cong \tch \otimes_{\ch \rtimes \Omega(\Xi'_0)(1)} (\Xi'_0 \otimes V) \\ &\cong \pi_{S_0, \G', \Xi', V}.\end{align*}


\subsection{} Let $q$ be an indeterminant. We denote by $\tch_q$ the generic pro-$p$ Hecke algebra $\tch(q,c)$ over $\kk[q]$ such that $q_t=q$ for all $t \in T(1)$. Notice that $\tch=\tch_q / q \tch_q$. Let $\tw \in \tW(1)$. There exists $w \in W_0(1)$, $\l_1, \l_2 \in X^+(1)$ such that $\tw=w \l_1 \l_2\i$. Following Vign\'{e}ras, we define $$E_{\tw}=q^{\frac{1}{2}(\ell(\mu_2)-\ell(\mu_1)-\ell(w)+\ell(\tw))} T_{w \l_1}T_{\l_2}\i \in \tch_q.$$ It is known that $E_{\tw}$ is independent of the choices of $w$, $\l_1$ and $\l_2$. By \cite{V05}, the set $\{E_{\tw}; \tw \in \tW\}$ forms a basis of $\tch_q$.

We say that $M \in R(\tch)_\kk$ is {\it supersingular} if $E_w M=0$ for $w \in \tW(1)$ with $\ell(w) \gg 0$. We have the following criterions on the supersingular representations.

\begin{thm}\label{5.4}
Let $M \in R(\tilde \ch)_\kk$. The following conditions are equivalent:

(1) $M$ is supersingular.

(2) $Tr(\overline{\tch}^{nss}, M)=0$.

(3) $M$ is spanned by $\pi_{S_0, \G, \Xi, V}$ for $(S_0, \G) \in \aleph^*$ and $(\Xi, V) \in \cp(S_0, \G)$ with $\sharp W_\G, \sharp W_{\G'}<\infty$, where $\G'=\{\pi(s); s \in S_{\aff}(1)\smallsetminus \G(1); \Xi(c_s) \neq 0\}$.
\end{thm}

\begin{rmk}
The equivalence between (1) and (3) is first obtained by Ollivier in \cite[Theorem 5.14]{O13} and Vign\'era in \cite[Theorem 6.18]{V14} if $\tch$ is the pro-$p$ Iwahari-Hecke algebra of a $p$-adic group.
\end{rmk}

\begin{lem}\label{bound}
Let $x, y \in \tW(1)$ with $\ell(x) \le \ell(y)$. Then
\begin{gather*} q^{\frac{1}{2}(\ell(x)-\ell(y)+\ell(yx))} T_y T_{x\i}\i \in \bigl(\oplus_{z \in \tW(1), \ell(z) \ge \frac{1}{2}(\ell(y)-\ell(x)+\ell(yx))} \ZZ T_z \bigr)+q \ZZ[q] \tch_q, \\
q^{\frac{1}{2}(\ell(y)-\ell(x)+\ell(xy))} T_x T_{y\i}\i \in \bigl(\oplus_{z \in \tW(1), \ell(z) \ge \frac{1}{2}(\ell(y)-\ell(x)+\ell(xy))} \ZZ {}^\iota T_z \bigr)+q \ZZ[q] \tch_q.
\end{gather*}
\end{lem}
\begin{proof}
We prove the first statement. The second one can be proved in the same way.

We argue by induction on $\ell(x)$. If $\ell(x)=0$, then statement is obvious. Assume $\ell(x) \ge 1$ and the statement holds for any $x'$ with $\ell(x') < \ell(x)$. Let $s \in S_{\aff}(1)$ such that $s x < x$.

If $y s < y$, then $$q^{\frac{1}{2}(\ell(x)-\ell(y)+\ell(yx))} T_y T_{x\i}\i=q^{\frac{1}{2}(\ell(sx)-\ell(ys)+\ell(ys sx))}T_{ys} T_{(sx)\i}\i$$ and $\ell(y)-\ell(x)+\ell(yx)=\ell(ys)-\ell(sx)+\ell(yssx)$. The statement follows from induction hypothesis.

If $y s >y$, then
\begin{align*} q^{\frac{1}{2}(\ell(x)-\ell(y)+\ell(yx))} T_y T_{x\i}\i= &q^{\frac{1}{2}(\ell(sx)-\ell(ys)+\ell(yx))} T_{ys} T_{(sx)\i}\i \\ &+q^{\frac{1}{2}(\ell(sx)-\ell(y)+\ell(yx)-1)}(1-q) T_y T_{(sx)\i}\i .
\end{align*}

By inductive hypothesis, $$q^{\frac{1}{2}(\ell(sx)-\ell(ys)+\ell(yx))} T_{ys} T_{(sx)\i}\i \in \bigl(\oplus_{z \in \tW(1), \ell(z) \ge \frac{1}{2}(\ell(y)-\ell(x)+\ell(xy))} \ZZ T_z \bigr)+q \ZZ[q] \tch_q.$$

Let $\a$ be the simple root associated to $s$ and $\b=x\i(\a)$. Then $\b<0$ since $s x <x$ and $yx(\b)=y(\a)>0$ since $y s>y$. Hence $y s x= y x s_{\b}<y x$. Therefore, $\ell(y s x) \le \ell(y x)-1$ and $\ell(sx)-\ell(y)+\ell(yx)-1 \ge \ell(sx)-\ell(y)+\ell(ysx)$.

If $\ell(y s x)<\ell(y x)-1$, then $\ell(sx)-\ell(y)+\ell(yx)-1>\ell(sx)-\ell(y)+\ell(ysx)$ and by inductive hypothesis, $q^{\frac{1}{2}(\ell(sx)-\ell(y)+\ell(yx)-1)}(1-q) T_y T_{(sx)\i}\i \in q \ZZ[q] \tch_q$ and the statement holds in this case.

If $\ell(y s x)=\ell(y x)-1$, then $\ell(y)-\ell(x)+\ell(yx) = \ell(y)-\ell(sx)+\ell(ysx)$ and by inductive hypothesis, \begin{align*} & q^{\frac{1}{2}(\ell(sx)-\ell(y)+\ell(yx)-1)}(1-q) T_y T_{(sx)\i}\i  \\ & \in \bigl(\oplus_{z \in \tW(1), \ell(z) \ge \frac{1}{2}(\ell(y)-\ell(sx)+\ell(y s x))} \ZZ T_z \bigr)+q \ZZ[q] \tch_q \\ &=\bigl(\oplus_{z \in \tW(1), \ell(z) \ge \frac{1}{2}(\ell(y)-\ell(x)+\ell(yx))} \ZZ T_z \bigr)+q \ZZ[q] \tch_q
\end{align*}
The statement also holds in this case.
\end{proof}

\begin{cor}\label{bound1}
Let $(S_0, \G) \in \aleph^*$. Let $w \in \tW(1)$ with $\ell(w)> 2 \sharp W_\G$. Then in $\tch=\tch_q / q \tch_q$, we have either $E_{w} \in \oplus_{z \in \tW(1), \supp(z) \nsubseteq \G} \kk \, T_z$ or $E_{\tw} \in \oplus_{z \in \tW(1), \supp(z) \nsubseteq \G} \kk \, {}^{\iota} T_z$.
\end{cor}
\begin{proof}
By definition, $E_{w}=q^{\frac{1}{2}(\ell(x)-\ell(y)+\ell(yx))} T_y T_{x\i}\i$ for some $x, y \in \tW(1)$ such that $yx=w$. Applying Lemma \ref{bound}, we see that $E_{w} \in \oplus_{z \in \tW(1), \ell(z) > \sharp W_\G} \kk T_z$ if $\ell(x) \le \ell(y)$ and $E_{w} \in \oplus_{z \in \tW(1), \ell(z) > \sharp W_\G} \kk {}^{\iota} T_z$ if $\ell(y) \le \ell(x)$. The statement follows by noticing that $\supp(z) \nsubseteq \G$ if $\ell(z)>\sharp W_\G$.
\end{proof}

\subsection{Proof of Theorem \ref{5.4}}
(1) $\Rightarrow$ (2). Let $w \in \tW(1)_{\min}$ with associated standard pair $(J, \G)$ such that $J=J_{\bar \nu_y} \subsetneq S_0$. It remains to show $Tr(T_{w}, M)=Tr({}^\iota T_{w\i}, M)=0$. By Proposition \ref{power}, there exists $x \in \Omega_J(1)$ such that $\nu_x=\bar \nu_y$ and $u \in W_\G(1)$ such that $T_{w}^n \equiv (T_{(u x)}^J)^n \mod [\tch, \tch]$ for $n \gg 0$. Note that $(T_{(u x)}^J)^n$ is a linear combination of $T_{u' x^n}^J$, where $u' \in W_\G(1)$. By definition, there exists a sufficiently large $m_0$ such that $x^{m_0} \in X^+(1)$ and $T_{x^{m_0}} M=E_{x^{m_0}} M=0$. By Proposition \ref{u-x}, $\ell(u' x^n)=\ell(u' x^{n-m_0}) +\ell(x^{m_0})$ for $u' \in W_\G(1)$ and $n \gg 0$. Thus $T_{u' x^n}^J M=T_{u' x^n} M=T_{u' x^{n-m_0}} T_{x^{m_0}} M=0$. Therefore, $Tr(T_w^n, M)=Tr((T_{u x}^J)^n, M)=0$ for $n \gg 0$. Hence $Tr(T_{w}, M)=0$ as desired. The equality $Tr({}^\iota T_{w \i}, M)=0$ follows in a similar way by noticing that ${}^\iota T_{x^{-m_0}}=E_{x^{-m_0}}$.

(2) $\Rightarrow$ (3). By Proposition \ref{rig}, $M$ and $\iota(M)$ lie in the $\kk$-span of $(\pi_{S_0, \G, \Xi, V})_{(S_0, \G) \in \aleph^*, (\Xi, V) \in \cp(S_0, \G)}$. Write $$\iota(M)=\sum_{(S_0, \G) \in \aleph^*/ \sim, \ (\Xi, V) \in \cp(S_0, \G)} a_{\G, \Xi, V} \pi_{S_0, \G, \Xi, V}$$ for some $a_{\G, \Xi, V} \in \kk$. By Proposition \ref{iota}, $$M=\sum_{(S_0, \G) \in \aleph^*/ \sim, \ (\Xi, V) \in \cp(S_0, \G)} a_{\G, \Xi, V} \pi_{S_0, \G'_{\G, \Xi}, \Xi'_{\G, \Xi}, V},$$ where $\G'_{\G, \Xi}=\{s \in S_{\aff} \smallsetminus \G; \Xi(c_s) \neq 0\}$ and $\Xi'_{\G, \Xi}$ is the character of $\ch_{\G'_{\Xi, \G}}$ defined by $\Xi'_{\G, \Xi} |_Z=\Xi|_Z$ and $\Xi'_{\G, \Xi}(s)=\Xi(c_s)$ for $s \in \G'_{\G, \Xi}$. By Lemma \ref{inf-rig}, $a_{\G, \Xi, V}=0$ unless $\sharp W_{\G'_{\G, \Xi}}<\infty$. Part (2) is proved.

(3) By definition, $T_x \, \pi_{S_0, \G, \Xi, V}={}^{\iota} T_x \, \pi_{S_0, \G', \Xi', V'}=0$ for any $x \in \tW(1)$ such that $\supp(x) \nsubseteq \G$ and $\supp(x) \nsubseteq \G'$. Applying Corollary \ref{bound1}, $E_{w} \, \pi_{S_0, \G, \Xi, V}=0$ for $w \in \tW(1)$ with $\ell(\tw)> 2 \max\{\sharp W_\G, \sharp W_{\G'}\}$.

\end{document}